\documentclass{amsart}

\usepackage{amsmath}
\usepackage{amssymb}
\usepackage{amsfonts}
\usepackage{amsbsy}
\usepackage{amscd}
\usepackage{eucal}
\usepackage{graphicx}

\numberwithin{equation}{section}

\newcommand{\beq}{\begin{equation}}
\newcommand{\eeq}{\end{equation}}

\newtheorem{theorem}{Theorem}[section]
\newtheorem{proposition}[theorem]{Proposition}

\newtheorem{corollary}[theorem]{Corollary}
\newtheorem{remark}{Remark}
\newtheorem{note}{Note}

\newcommand{\rmd}{\mathrm{d}}
\newcommand{\rmi}{\mathrm{i}}
\newcommand{\Real}{\mathop{\mathrm{Re}}}
\newcommand{\Imag}{\mathop{\mathrm{Im}}}

\DeclareMathOperator{\sgn}{sgn}

\newcommand{\N}{\mathbb{N}}
\newcommand{\R}{\mathbb{R}}
\newcommand{\Z}{\mathbb{Z}}
\newcommand{\C}{\mathbb{C}}

\newcommand{\tfJ}{\widetilde{\mathfrak{J}}}

\newcommand{\sds}{\strut\displaystyle}
\newcommand{\ud}{\frac{1}{2}}
\newcommand{\uds}{\textstyle\frac{1}{2}}

\newcommand{\sku}{\vspace*{0.1cm}}
\newcommand{\skd}{\vspace*{0.2cm}}
\newcommand{\skt}{\vspace*{0.3cm}}

\begin{document}

\title[Bessel functions of the first kind]
{Integral representation for Bessel's functions of the first kind and Neumann series}

\author[E. De Micheli]{Enrico De Micheli}
\address{Consiglio Nazionale delle Ricerche, Via De Marini, 6 - 16149 Genova, Italy}
\email{enrico.demicheli@cnr.it}



\begin{abstract}
A Fourier-type integral representation for Bessel's function of the first kind and complex order is
obtained by using the Gegenbuaer extension of Poisson's integral representation for the Bessel function
along with a trigonometric integral representation of Gegenbauer's polynomials. This representation lets us
express various functions related to the incomplete gamma function in series of Bessel's functions.
Neumann series of Bessel functions are also considered and a new closed-form integral representation for this 
class of series is given. The density function of this representation is simply the analytic function 
on the unit circle associated with the sequence of coefficients of the Neumann series.
Examples of new closed-form integral representations of special functions are also presented. 
\end{abstract}

\maketitle

\section{Introduction}
\label{se:introduction}

The Bessel function of the first kind and order $\nu$ is defined by the series \cite[Eq. 8, p. 40]{Watson}
\beq
J_\nu(z) \doteq \left(\frac{z}{2}\right)^\nu\sum_{k=0}^\infty\frac{(-1)^k\,(z/2)^{2k}}{k! \, \Gamma(\nu+k+1)},
\qquad (z\in\C;\nu\in\C,\nu\neq-1,-2,\ldots),
\label{intro.1}
\eeq
which is convergent absolutely and uniformly in any closed domain of $z$ and in any bounded
domain of $\nu$.
It is the solution of Bessel's equation
\beq
z^2\,\frac{\rmd^2 y}{\rmd z^2}+z\,\frac{\rmd y}{\rmd z}+(z^2-\nu^2)=0,
\label{intro.2}
\eeq
which is nonsingular at $z=0$. The function $J_\nu(z)$ is therefore an analytic function of $z$ for any $z$,
except for the branch point $z=0$ if $\nu$ is not an integer.

\sku

In this paper, a new integral representation for the Bessel functions of the first kind and order $\nu\in\C$,
$\Real\nu>-\ud$, is obtained. This representation, which is given in Section \ref{se:besselfunctions},
generalizes to complex values of the order $\nu$ the classical Bessel's integral
\beq
J_n(z) = \frac{(-\rmi)^n}{\pi}\int_0^\pi e^{\rmi z\cos\theta}\cos n\theta\,\rmd\theta
\qquad (n\in\Z),
\label{add.1}
\eeq
which holds only for integral order values (see, e.g., formula \eqref{Jcos.3}).
The Bessel functions of the first kind are expressed in terms of an integral 
of Fourier-type, which involves the regularized incomplete gamma function $P(\{\nu\},z)$, with $\{\nu\}$
being the (complex) fractional part of $\nu$. This result is achieved by using the Poisson integral representation 
of $J_{\nu+n}(z)$ ($\Real\nu>-\ud$, $\ell=0,1,2,\ldots$)
given in terms of Gegenbauer polynomials and then by exploiting the integral representation of the latter
polynomials that we prove in Section \ref{se:dimension}. 

The Fourier form of the integral representation of $J_{\nu+n}(z)$ motivates us to consider
the Fourier inversion formula in Section \ref{se:inversion}. This analysis
allows us to obtain, for either integer and half-integer orders $\nu$, 
a trigonometric expansion of the lower incomplete gamma function, whose 
coefficients are related to (modified and unmodified) Bessel functions of the first kind. Thus,
classical Bessel series expansions of various special functions 
which are connected with the incomplete gamma function, can be easily obtained in a unified form. 

A Neumann series of Bessel functions is an expansion of the type 
\beq
\sum_{n=0}^\infty a_n \,J_{\nu+n}(z) \qquad (z,\nu\in\C),
\label{Neu}
\eeq
where $\{a_n\}_{n=0}^\infty$ are given coefficients. Many special
functions of the mathematical physics enjoy an expansion of this type, e.g., 
Kummer confluent hypergeometric's, Lommel's , Kelvin's, Whittaker's, and so on.
In Section \ref{se:neumann} we exploit the novel representation of Bessel's functions
presented in Section \ref{se:besselfunctions} in order to derive a simple closed-form integral 
representation of expansions \eqref{Neu}. The set of coefficients $\{a_n\}_{n=0}^\infty$, which
characterizes the series \eqref{Neu}, comes into this integral representation in a very simple way
through the associated analytic function $A(z)$ in the unit disk $\mathbb{D}$,
i.e.: $A(z)=\sum_{n=0}^\infty a_n z^n$ with a weak restriction on the 
sequence of coefficients $\{a_n\}$, which is sufficient to guarantee the boundedness of $A(z)$ on
the unit circle. Examples of special functions with this kind of closed-form integral representation are also given. 

\section{Connection between the coefficients of Gegenbauer and Fourier expansions}
\label{se:dimension}

The Gegenbauer (ultraspherical) polynomials $C^\nu_\ell(x)$ of order $\nu$ ($\nu>-\ud, \nu\neq 0$)
may be defined by means of the generating function \cite[Eq. (4.7.23)]{Szego}
(see also \cite[Eq. (4.7.6)]{Szego}):
\begin{equation}
\sum_{\ell=0}^\infty C_\ell^{\nu}(x)\,t^\ell = (1-2xt+t^2)^{-\nu} \qquad 
(x\in[-1,1];\,t\in(-1,1)).
\label{a.1}
\end{equation}
For fixed $\nu$, the Gegenbauer polynomials are orthogonal on the interval $[-1,1]$ 
with respect to the weight function $w^{\nu}(x)=(1-x^2)^{(\nu-1/2)}$, that is:
\beq
\int_{-1}^1 C_\ell^{\nu}(x) \, C_n^{\nu}(x)\, w^{\nu}(x) \,\rmd x = h_\ell^\nu\ \delta_{\ell,n},
\label{a.9}
\eeq
where:
\beq
h_\ell^\nu = \frac{1}{(\ell+\nu)}\frac{2^{1-2\nu}\,\pi\,\Gamma(\ell+2\nu)}{\ell!\,\left[\Gamma(\nu)\right]^2}.
\label{a.9bis}
\eeq
Particularly important special cases of Gegenbauer polynomials are obtained
for $\nu=1/2$, which gives the Legendre polynomials $P_\ell(x)$,
for $\nu=1$ yielding the Chebyshev polynomials of the second kind $U_\ell(x)$, and in a suitable limiting
form for $\nu\to 0$ the Chebyshev polynomials of the first kind $T_\ell(x)$ \cite{Szego}:
$\lim_{\nu\to 0}(\ell!/(2\nu)_\ell) C_\ell^{\nu}(x) = T_\ell(x)$, where $(\cdot)_\ell$ denotes the Pochhammer symbol.

For $\Real\nu>0$ they admit the following integral representation \cite[Eq. (1), p. 559]{Vilenkin}:
\beq
C_\ell^{\nu}(x)=\frac{(\ell+\nu)\,h_\ell^\nu}{\pi}
\int_0^\pi (x+\rmi\sqrt{1-x^2}\cos\eta)^\ell\,(\sin\eta)^{2\nu-1}\,\rmd\eta. 
\label{d.1}
\eeq

Now, our aim in this section is to present an integral representation of the Gegenbauer polynomials.
To this end, we prove the following proposition (see also \cite{Bros,DeMicheli1}).

\begin{proposition}
\label{pro:1}

For $\Real\nu>0$, the following integral representation for the Gegenbauer polynomials 
$C_\ell^{\nu}(\cos u)$ ($u\in[0,\pi]$, $\ell\in\N_0$) holds:
\beq
C_\ell^{\nu}(\cos u) = \frac{(\ell+\nu)\,h_\ell^\nu}{\pi}
\frac{2^{\nu-1}e^{-\rmi\pi\nu}}{(\sin u)^{(2\nu-1)}}
\int_u^{2\pi-u} \!\! e^{\rmi(\ell+\nu)t} \ (\cos u-\cos t)^{\nu-1}\,\rmd t.
\label{d.22}
\eeq
\end{proposition}

\begin{proof}
In representation \eqref{d.1} we first introduce the variable $u$, defined by $x=\cos u$ ($u\in[0,\pi]$):
\beq
C_\ell^{\nu}(\cos u)=\frac{(\ell+\nu)\,h_\ell^\nu}{\pi}
\int_0^\pi (\cos u+\rmi\sin u\cos\eta)^\ell\,(\sin\eta)^{2\nu-1}\,\rmd\eta.
\label{d.1.bis}
\eeq
\begin{figure}[tb]
\begin{center}
\includegraphics[width=3in,angle=0]{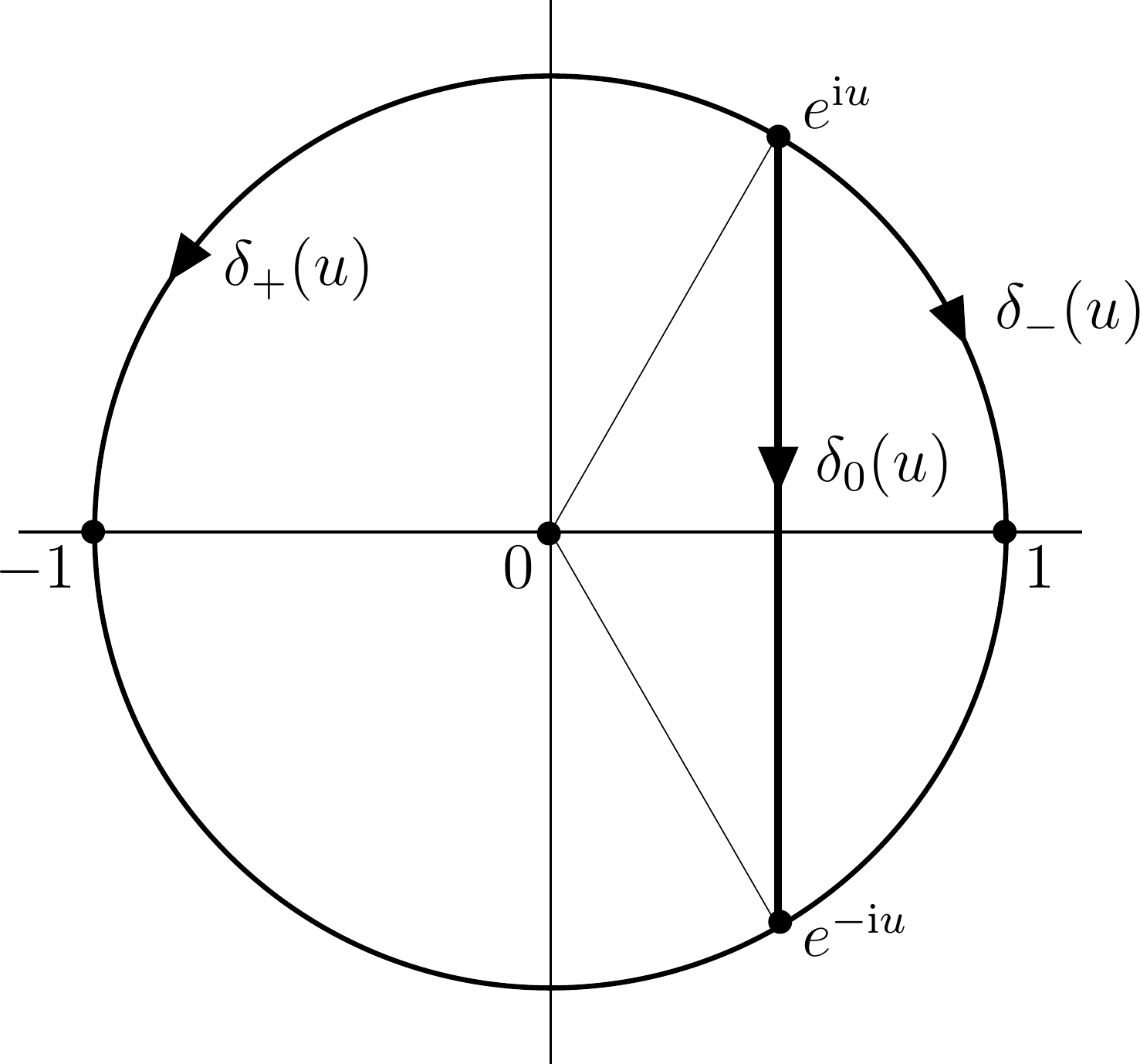}
\caption{\label{figu1} Integration path for the evaluation of the integral representation
(\protect\ref{d.22}) of the Gegenbauer polynomials.}
\end{center}
\end{figure}
Then, in the integral in \eqref{d.1.bis} we substitute to $\eta$ the complex integration variable $\tau$ defined by
\begin{displaymath}
e^{\rmi\tau}=\cos u + \rmi\sin u\cos\eta.
\label{d.p1}
\end{displaymath}
It can be checked that
\beq
2e^{\rmi\tau}(\cos\tau-\cos u)=(e^{\rmi\tau}-e^{\rmi u})(e^{\rmi\tau}-e^{-\rmi u})=\sin^2\! u \, \sin^2\!\eta.
\label{d.p2}
\eeq
Now, since $e^{\rmi\tau}\rmd\tau=-\sin u\sin\eta\,\rmd\eta$, the integrand on the right-hand side of \eqref{d.1.bis} 
can be written as follows:
\beq
\begin{split}
-(\sin u)^{-(2\nu-1)}\,e^{\rmi(\ell+1)\tau}\,
\left[(e^{\rmi\tau}-e^{\rmi u})(e^{\rmi\tau}-e^{-\rmi u})\right]^{(\nu-1)}\,\rmd\tau \\
\quad =-(\sin u)^{-(2\nu-1)}\,e^{\rmi(\ell+\nu)\tau}\,\left[2(\cos\tau-\cos u)\right]^{(\nu-1)}\,\rmd\tau.
\end{split}
\label{d.p3}
\eeq
In order to determine the integration path, consider the intermediate step where $e^{\rmi\tau}$ 
is chosen as integration variable; the original path (corresponding to $\eta\in[0,\pi]$) is the 
(oriented) linear segment $\delta_0(u)$ starting at $e^{\rmi u}$ and ending at $e^{-\rmi u}$. 
Since (as shown by \eqref{d.p3}) the integrand is an analytic function of $e^{\rmi\tau}$ in the 
disk $|e^{\rmi\tau}|<1$ (since $\ell\in\N_0$), the integration path $\delta_0(u)$ can be replaced by the circular path 
$\delta_+(u)=\{e^{\rmi\tau}:\,\tau=t,u\leqslant t\leqslant 2\pi-u\}$ (see Fig. \ref{figu1}). 
Moreover, by using the fact that
$[(e^{\rmi\tau}-e^{\rmi u})(e^{\rmi\tau}-e^{-\rmi u})]^{\nu-1}$ is positive for $e^{\rmi\tau}\in\delta_0(u)\cup\R$
and therefore at $e^{\rmi\tau}=e^{\rmi\pi}$, 
we conclude from the left equality in \eqref{d.p2} that in 
the r.h.s. of \eqref{d.p3} the following specification holds (for $\tau=t; u\leqslant t\leqslant 2\pi-u$):
\begin{displaymath}
[2(\cos t-\cos u)]^{\nu-1}
= e^{-\rmi(\nu-1)\pi} [2(\cos u-\cos t)]^{\nu-1}.
\label{d.p4}
\end{displaymath}
Finally, accounting for this latter expression, the integral representation \eqref{d.22} follows directly 
from \eqref{d.1.bis}.
\end{proof}

\begin{remark}
In the proof of Proposition \ref{pro:1} we could even choose the integration path
$\delta_-(u) = \{e^{\rmi\tau}:\,\tau=t,-u\leqslant t\leqslant u\}$, oriented as shown in Fig. \ref{figu1}.
Following the same arguments given above, it is easily seen that this option yields the following 
alternative integral representation of the Gegenbauer polynomials \cite{Bros,DeMicheli3}:
\beq
C_\ell^{\nu}(\cos u) = \frac{(\ell+\nu)\,h_\ell^\nu}{\pi}
\frac{2^{\nu-1}}{(\sin u)^{2\nu-1}}
\int_{-u}^{u} \!\! e^{\rmi(\ell+\nu)t} \ (\cos t-\cos u)^{\nu-1}\,\rmd t.
\label{gege}
\eeq
\end{remark}

\skt

\section{Fourier-type integral representation of Bessel functions of the first kind}
\label{se:besselfunctions}

It is well-known that the Bessel functions of the first kind of integral order $n$ can be represented 
by Bessel's trigonometric integral \eqref{add.1} (see also \cite[Eq. 10.9.2]{DLMF}).
For $n\in\Z$, $J_n(x)$ satisfies the symmetry relation: $J_{-n}(x) = (-1)^n\,J_n(x)$.
This latter relation no longer holds when $n$ is not an integral number since, for $\nu\not\in\Z$, 
$J_{\nu}(x)$ and $J_{-\nu}(x)$ are linearly independent solutions of the Bessel equation of order $\nu$.
Numerous formulae express the Bessel functions of the first kind as definite integrals, which can
be exploited to obtain, for instance, approximations and asymptotic expansions 
(see \cite[Chapter VI]{Watson}, \cite{Drascic},
and the website \cite[Sect. 10.9]{DLMF} for a useful collection of formulae).
Our goal in what follows is to generalize the trigonometric representation \eqref{add.1} in order
to obtain a Fourier-type integral representation for the Bessel functions of the first kind with holds
for complex order $\nu$ with $\Real\nu>-\ud$. We can prove the following theorem.

\begin{theorem}
\label{pro:J}

Let $\ell\in\N_0\doteq\{0,1,2,\ldots\}$ be a non-negative integer, and $\nu$ any complex number such that $\Real\nu>-\ud$.
Then, the following integral representation for the Bessel functions of the first kind $J_{\nu+\ell}(z)$ holds:
\beq
J_{\nu+\ell}(z) = (-\rmi)^\ell \int_{-\pi}^\pi \mathfrak{J}^{(\nu)}_z(\theta) \ e^{\rmi\ell\theta}\,\rmd\theta 
\qquad (\ell\in\N_0,\Real\nu>-\uds),
\label{d.J.1}
\eeq
where the $2\pi$-periodic function $\mathfrak{J}^{(\nu)}_z(\theta)$ is given by
\beq
\mathfrak{J}^{(\nu)}_z(\theta) = \frac{\rmi^\nu}{2\pi} e^{\rmi\nu[\theta-\pi\sgn(\theta)]}e^{\rmi z\cos\theta} 
\ P(\{\nu\},-\rmi z\,(1-\cos\theta))  \qquad (\Real\nu>-\uds),
\label{d.J.2}
\eeq
and $\sgn(\cdot)$ is the sign function,$\{\nu\}$ is the complex fractional part of $\nu$,
$P(\nu,w)\doteq\gamma(\nu,w)/\Gamma(\nu)$ denotes the regularized incomplete gamma function, 
$\gamma(\nu,w)$ being the lower incomplete gamma function.
\end{theorem}

\begin{proof}
We start from Gegenbauer's generalization of Poisson's integral representation of the 
Bessel functions of the first kind \cite[\S 3.32, Eq. (1), p. 50]{Watson}:
\beq
J_{\nu+\ell}(z) = \frac{(-\rmi)^\ell\Gamma(2\nu)\,\ell!}{\Gamma(\nu+\uds)\Gamma(\uds)\Gamma(2\nu+\ell)}
\left(\frac{z}{2}\right)^\nu\int_0^\pi e^{\rmi z\cos u}(\sin u)^{2\nu} 
C_\ell^\nu(\cos u)\,\rmd u,
\label{d.J.4}
\eeq
where $C_\ell^\nu(t)$ is the Gegenbauer polynomial of order $\nu$ and degree $\ell$.
Formula \eqref{d.J.4} holds for $\ell\in\N_0$ and $\Real\nu>-\uds$.
Now, we plug \eqref{d.22} into \eqref{d.J.4} and, using the Legendre duplication formula
for the gamma function, we obtain:
\beq
J_{\nu+\ell}(z) = \frac{(-\rmi)^\ell\,e^{-\rmi\pi\nu}\, z^\nu}{2\pi\,\Gamma(\nu)}
\int_0^\pi\!\rmd u\,\sin u\,e^{\rmi z\cos u}
\int_u^{2\pi-u}\!\!e^{\rmi(\ell+\nu)\theta}\,(\cos u-\cos\theta)^{\nu-1}\,\rmd\theta.
\label{d.J.6}
\eeq
Interchanging the order of integration, \eqref{d.J.6} can be written as follows:
\beq
\begin{split}
J_{\nu+\ell}(z) &= \frac{(-\rmi)^\ell\,e^{-\rmi\pi\nu}z^\nu}{2\pi\,\Gamma(\nu)}
\left[\int_{0}^\pi\!\rmd\theta\,e^{\rmi(\ell+\nu)\theta}\int_{0}^\theta \!
e^{\rmi z\cos u} (\cos u-\cos\theta)^{\nu-1}\,\sin u\,\rmd u \right. \\
&\quad\left. +\int_\pi^{2\pi}\!\rmd\theta\,e^{\rmi(\ell+\nu)\theta}
\int_{0}^{2\pi-\theta}\!e^{\rmi z\cos u}\,(\cos u-\cos\theta)^{\nu-1}\,\sin u\,\rmd u \right].
\end{split}
\label{d.J.7}
\eeq
Next, changing the integration variables: $\theta-2\pi\to\theta$ and $u\to-u$,
the second integral on the r.h.s. of \eqref{d.J.7} becomes:
\beq
e^{\rmi\nu 2\pi}\int_{-\pi}^0 \rmd\theta \ e^{\,\rmi(\ell+\nu)\theta}
\int_0^\theta e^{\rmi z\cos u}(\cos u-\cos\theta)^{\nu-1} \sin u \,\rmd u,
\label{d.J.7.1} \nonumber
\eeq
which, inserted in \eqref{d.J.7}, yields:
\beq
\begin{split}
J_{\nu+\ell}(z) &= \frac{(-\rmi)^\ell z^\nu}{2\pi\,\Gamma(\nu)}
\left[e^{\rmi\nu\pi}\int_{-\pi}^0\!\rmd\theta\,e^{\,\rmi(\ell+\nu)\theta}
\int_{0}^\theta\!e^{\rmi z\cos u} (\cos u-\cos\theta)^{\nu-1}\sin u\,\rmd u \right. \\
&\quad\left. +e^{-\rmi\nu\pi}\int_0^{\pi}\!\rmd\theta\,e^{\rmi(\ell+\nu)\theta}\int_{0}^{\theta} \! 
e^{\rmi z\cos u} (\cos u-\cos\theta)^{\nu-1}\sin u\,\rmd u \right].
\end{split}
\label{d.J.8}
\eeq
Formula \eqref{d.J.8} can then be written as follows:
\beq
J_{\nu+\ell}(z) = (-\rmi)^\ell
\int_{-\pi}^\pi \mathfrak{J}^{(\nu)}_z(\theta) \ e^{\rmi\ell\theta}\,\rmd\theta 
\qquad (\ell\in\N_0, \Real\nu>0),
\label{d.J.9}
\eeq
where
\beq
\begin{split}
\mathfrak{J}^{(\nu)}_z(\theta) &= 
\frac{z^\nu}{2\pi\Gamma(\nu)}e^{\rmi\nu[\theta-\pi\sgn(\theta)]} 
\int_{\cos\theta}^1 e^{\rmi zt}(t-\cos\theta)^{\nu-1}\,\rmd t \qquad (\Real\nu>0).
\end{split}
\label{d.J.10}
\eeq
Then, changing in \eqref{d.J.10} the integration variable $-\rmi z(t-\cos\theta)\to t$
and recalling that for $\Real\nu>0$, $\gamma(\nu,w) = \int_0^w t^{\nu-1}\,e^{-t}\,\rmd t$,
$\mathfrak{J}^{(\nu)}_z(\theta)$ can be rewritten as follows:
\beq
\mathfrak{J}^{(\nu)}_z(\theta) = \frac{\rmi^\nu}{2\pi} e^{\rmi\nu[\theta-\pi\sgn(\theta)]}e^{\rmi z\cos\theta} 
\ P(\nu,-\rmi z\,(1-\cos\theta))  \qquad (\Real\nu>0).
\label{d.J.10.bis}
\eeq
The regularized gamma function $P(\nu,w)$ enjoys the following recurrence relation \cite[Eq. 8.8.11]{DLMF}:
\beq
P(\nu,w) = P(\nu-k,w)-e^{-w}\sum_{j=1}^k \frac{w^{\nu-j}}{\Gamma(\nu-j+1)}
\qquad (k=1,2,\ldots).
\label{recurrence}
\eeq
Let us set $k$ as the integer part of $\Real\nu$, i.e.: $k=[\Real\nu]$, where
\beq
[x] \doteq 
\begin{cases}
\lfloor x \rfloor \quad (\mathrm{the \ greatest \ integer} \leqslant x) & \mathrm{for} \; x \geqslant 0, \\[+3pt]
\lceil x \rceil \quad (\mathrm{the \ smallest \ integer} \geqslant x) & \mathrm{for} \; x < 0.
\end{cases}
\label{integerpart}
\eeq
Moreover, we denote by $\{\nu\}$ the complex fractional part of $\nu$, 
defined as: $\{\nu\}\doteq\nu-[\Real\nu]$, with $-1<\Real\{\nu\}<1$. 
We can then insert formula \eqref{recurrence} into \eqref{d.J.10.bis} and obtain:
\beq
\begin{split}
& \mathfrak{J}^{(\nu)}_z(\theta) = \frac{\rmi^\nu}{2\pi} e^{\rmi\nu[\theta-\pi\sgn(\theta)]} \\
&\quad\cdot\left[e^{\rmi z\cos\theta}P(\{\nu\},-\rmi z\,(1-\cos\theta))
-e^{\rmi z}\sum_{1\leqslant j\leqslant[\Real\nu]}
\frac{(-\rmi z)^{\nu-j}}{\Gamma(\nu-j+1)}(1-\cos\theta)^{\nu-j}\right].
\end{split}
\label{d.J.10.tris}
\eeq
When plugged into \eqref{d.J.9} the second term in the squared brackets of \eqref{d.J.10.tris}
(which is different from zero only if $[\Real\nu]\geqslant 1$) contributes
with a (finite) linear combination of the following integrals:
\beq
\int_{-\pi}^\pi e^{\rmi\nu[\theta-\pi\sgn(\theta)]}(1-\cos\theta)^{\nu-j}\,e^{\rmi\ell\theta}\,\rmd\theta,
\label{integrals}
\eeq
which are null for $\ell\in\N_0$ and $\Real\nu-j\geqslant0$. Therefore, $ \mathfrak{J}^{(\nu)}_z(\theta)$
finally reads
\beq
\mathfrak{J}^{(\nu)}_z(\theta) = \frac{\rmi^\nu}{2\pi} e^{\rmi\nu[\theta-\pi\sgn(\theta)]}
e^{\rmi z\cos\theta}P(\{\nu\},-\rmi z\,(1-\cos\theta)) \qquad (\Real\nu>0).
\label{d.J.10.quater}
\eeq
The $\nu$-domain of representation \eqref{d.J.9} can be analytically extended to the half-plane
$\Real\nu>-\ud$ where, for any $z\in\C\setminus(-\infty,0]$, the integrand is analytic on $\theta\in(-\pi,\pi]$ 
and the integral defines a function which is locally bounded on every compact subsets of $\Real\nu>-\ud$.
To see this, it is useful to make explicit the singularity brought by
the regularized gamma function $P(\nu,w)$ by writing the latter in terms of Tricomi's form
$\gamma^*(\nu,w)$ of the incomplete gamma function, i.e.: 
$P(\nu,w)=w^\nu\,\gamma^*(\nu,w)$, where
\beq
\gamma^*(\nu,w) \doteq e^{-w} \sum_{n=0}^\infty \frac{w^n}{\Gamma(\nu+n+1)}
\label{tricogamma}
\eeq
is an entire function in $w$ as well as in $\nu$ \cite[Eq. (2)]{Tricomi}. 
From \eqref{d.J.9} and \eqref{d.J.10.quater} we thus obtain:
\beq
\begin{split}
& J_{\nu+\ell}(z)=\frac{\rmi^{([\Real\nu]-\ell)}\,z^{\{\nu\}}}{2\pi} \\
&\quad\cdot\int_{-\pi}^\pi e^{-\rmi\nu\pi\sgn(\theta)} \, e^{\rmi z\cos\theta}(1-\cos\theta)^{\{\nu\}} \,
\gamma^*(\{\nu\},-\rmi z(1-\cos\theta)) \, e^{\rmi(\ell+\nu)\theta}\,\rmd\theta.
\end{split}
\label{d.J.1.bis}
\eeq
For negative values of $\Real\{\nu\}$, the integrand in \eqref{d.J.1.bis} has a singularity in $\theta=0$
due to the term $(1-\cos\theta)^{\{\nu\}}$, which is integrable for $\Real\{\nu\}>-\ud$. 
Hence, the integral \eqref{d.J.1.bis} defines 
an analytic function of $z$ on $\C$ and representation \eqref{d.J.1} holds true for $\Real\nu>-\ud$.
\end{proof}

\begin{note}
Alternatively, if in the Abel-type integral in \eqref{d.J.10} we change the integration 
variable $[2z(t-\cos\theta)/\pi]^\nu\to t$ the function $\mathfrak{J}^{(\nu)}_z(\theta)$ can be 
written in terms of Fresnel-type integrals $\mathcal{F}_\lambda(z)$, i.e:
\beq
\mathfrak{J}^{(\nu)}_z(\theta) = \frac{\pi^{\nu-1}}{2^{\nu+1}\Gamma(\nu+1)}
e^{\rmi\nu[\theta-\pi\sgn(\theta)]}\,e^{\rmi z\cos\theta} \ 
\mathcal{F}_{1/\nu}\left(\left[\frac{2z(1-\cos\theta)}{\pi}\right]^\nu\right),
\label{d.J.11}
\eeq
where 
\beq
\mathcal{F}_\lambda(w) \doteq \int_0^w \exp\left(\rmi\tfrac{\pi}{2} s^\lambda\right)\,\rmd s,
\label{d.J.11.2}
\eeq
and the classical Fresnel integrals $C(w)$ and $S(w)$ can be obtained as real and imaginary part
of \eqref{d.J.11.2} with $\lambda=2$, respectively.
\end{note}

\sku

Representation \eqref{d.J.1} can be reformulated as follows:
\beq
\begin{split}
\rmi^\ell J_{\nu+\ell}(z) &= \frac{\rmi^\nu}{2\pi}\left[
\int_{-\pi}^0 e^{\rmi z\cos\theta} P(\{\nu\},-\rmi z(1-\cos\theta))e^{\rmi((\ell+\nu)\theta+\pi\nu)}
\,\rmd\theta\right.\\
&\quad\qquad\left. +\int_0^\pi e^{\rmi z\cos\theta} P(\{\nu\},-\rmi z(1-\cos\theta))
e^{\rmi((\ell+\nu)\theta-\pi\nu)}\,\rmd\theta\right] \\
&=\frac{\rmi^\nu}{\pi}\int_0^\pi e^{\rmi z\cos\theta} 
P(\{\nu\},-\rmi z(1-\cos\theta)) \cos(\ell\theta+\nu(\theta-\pi))\,\rmd\theta.
\label{Jcos.1}
\end{split}
\eeq
In the last integral we now change the integration variable: $\theta\to\pi-\theta$, and obtain
for $\ell\in\N_0,\Real\nu>-\uds$:
\beq
J_{\nu+\ell}(z) = \frac{\rmi^{(\nu+\ell)}}{\pi} 
\int_0^\pi e^{-\rmi z\cos\theta} P(\{\nu\},-\rmi z(1+\cos\theta))\cos((\ell+\nu)\theta)\,\rmd\theta.
\label{Jcos.2}
\eeq
Now, we can put $\mu=\ell+\nu$, with $\Real\mu>-\ud$. Since $\{\mu\}=\{\nu\}$,
from \eqref{Jcos.2} we therefore obtain the following integral representation of the Bessel function 
of the first kind:
\beq
J_\mu(z) = \frac{\rmi^\mu}{\pi} 
\int_0^\pi e^{-\rmi z\cos\theta} \, P(\{\mu\},-\rmi z(1+\cos\theta))\,\cos\mu\theta \,\rmd\theta
\qquad (\Real\mu>-\uds).
\label{Jcos.3}
\eeq

As mentioned earlier, formula \eqref{Jcos.3} (and formulae \eqref{d.J.1} and \eqref{d.J.2} as well)
generalizes to complex values of $\mu$ $(\Real\mu>-\ud)$ the classical Bessel integral \eqref{add.1},
which holds for integral values of $m$.
In fact, if in \eqref{Jcos.3} we put $\mu\equiv m$ integer, then $\{\mu\}=0$ and, since
$\gamma^*(0,w) = 1$ (see \cite[Eq. (2.2)]{Gautschi}), formula \eqref{add.1} readily follows.

In view of the relation: $I_\mu(w) = (-\rmi)^\mu \,J_\mu(\rmi w)$ for 
$-\pi\leqslant\mathrm{ph}\,w\leqslant\frac{\pi}{2}$ \cite[Eq. 10.27.6]{DLMF}, we have 
also the following integral representation for the modified Bessel function of the first kind $I_\mu(z)$.
\begin{corollary}
\label{pro:modified}
Let $\mu$ be any complex number such that $\Real\mu>-\ud$. 
Then the following integral representation for the modified
Bessel function of the first kind holds ($-\pi\leqslant\mathrm{ph}\,z\leqslant\frac{\pi}{2}$):
\beq
I_{\mu}(z) = \frac{1}{\pi} \int_{0}^\pi e^{z\cos\theta} \, P(\{\mu\},z(1+\cos\theta))\,\cos\mu\theta\,\rmd\theta
\qquad (\Real\mu>-\uds).
\label{mod.1}
\eeq
\end{corollary}
Formula \eqref{mod.1} generalizes to complex values of $\mu$, $\Real\mu>-\ud$, the well-known representation
\cite[Eq. 10.32.3]{DLMF}:
\beq
I_m(z) = \frac{1}{\pi}\int_{0}^\pi e^{z\cos\theta}\, \cos m\theta\,\rmd\theta
\qquad (m\in\Z),
\label{mod.3}
\eeq
which holds for integral values of $m$.

\section{Inversion of the Fourier representation}
\label{se:inversion}

Representation \eqref{d.J.1} shows that, for fixed $z$, the function $\rmi^\ell J_{\nu+\ell}(z)$ 
coincides with the $\ell$th Fourier coefficient ($\ell\geqslant 0$) of the $2\pi$-periodic function 
$\mathfrak{J}^{(\nu)}_z(\theta)$. We are thus prompted to consider the following trigonometrical series:
\beq
\mathfrak{J}^{(\nu)}_z(\theta)=\frac{1}{2\pi}\sum_{\ell=-\infty}^\infty \tfJ_\ell^{(\nu)}(z) \ e^{-\rmi\ell\theta},
\label{inv.1}
\eeq
where $\tfJ_\ell^{(\nu)}(z)$ denote, for fixed $z$, 
the $\ell$th Fourier coefficients of $\mathfrak{J}^{(\nu)}_z(\theta)$:
\beq
\tfJ_\ell^{(\nu)}(z) = \int_{-\pi}^\pi \mathfrak{J}^{(\nu)}_z(\theta) \, e^{\rmi\ell\theta}\,\rmd\theta 
\qquad (\ell\in\Z).
\label{d.J.13}
\eeq
However, the Fourier sum representation \eqref{inv.1} of $\mathfrak{J}^{(\nu)}_z(\theta)$ can actually be 
written explicitly only in some specific cases for the lack of knowledge of the Fourier coefficients with $\ell<0$.
Indeed, equation \eqref{d.J.1} states that $\tfJ_\ell^{(\nu)}(z)=\rmi^\ell J_{\nu+\ell}(z)$ only 
for $\ell\geqslant 0$.
However, from equation \eqref{d.J.2} we see that (for fixed $z$) 
$\mathfrak{J}^{(\nu)}_z(\theta)$ enjoys the following symmetry:
\beq
\mathfrak{J}^{(\nu)}_z(-\theta) = e^{-\rmi 2\nu[\theta-\pi\sgn(\theta)]}\,\mathfrak{J}^{(\nu)}_z(\theta),
\label{d.J.12}
\eeq
which, substituted in \eqref{d.J.13}, gives
\beq
\tfJ_\ell^{(\nu)}(z) = \int_{-\pi}^\pi \mathfrak{J}^{(\nu)}_z(\theta) \, e^{\rmi 2\nu\pi\sgn(\theta)}
\,e^{-\rmi(\ell+2\nu)\theta}\,\rmd\theta.
\label{inv.2}
\eeq
Formula \eqref{inv.2} can thus induce a $\ell$-index symmetry on the Fourier coefficients $\tfJ_\ell^{(\nu)}(z)$ 
only if $2\nu$ is integer, i.e.: $\nu$ can be either integer or half-integer ($\nu>-\ud$).

\skd

\noindent\textit{Case $\nu$ integer.}
If $\nu\equiv n$ is a nonnegative integer, formula \eqref{inv.2} yields: 
\beq
\tfJ_{\ell}^{(n)}(z) = \tfJ_{-\ell-2n}^{(n)}(z) \qquad (\ell\in\Z; n=0,1,2,\ldots).
\label{d.J.14}
\eeq
For $n=0$ all the Fourier coefficients $\tfJ_{\ell}^{(n)}(z)$, $\ell\in\Z$, can be determined directly
from \eqref{d.J.1} since $\tfJ_{\ell}^{(n)}(z) = \tfJ_{-\ell}^{(n)}(z)$.
For $n\neq 0$, the Fourier coefficients $\tfJ_{\ell}^{(n)}(z)$ with $-2n+1\leqslant \ell\leqslant -1$
cannot be obtained by means of the above symmetry relation, but can nevertheless be expressed as 
(linear) functionals of the set $\{\tfJ_{\ell}^{(n)}(z);\ell\in\N_0\}$ \cite{Bros}.
Since $\{\nu\}=0$ and $P(0,w)=1$, from \eqref{d.J.2} we have: 
$2\pi\mathfrak{J}^{(n)}_z(\theta)=(-\rmi)^n \, e^{\rmi z\cos\theta}\,e^{\rmi n\theta}$, which inserted
in \eqref{inv.1}, and using \eqref{d.J.14}, gives:
\beq
(-\rmi)^n\,e^{\rmi z\cos\theta} 
=\tfJ_{-n}^{(n)}(z)+2\sum_{\ell=-n+1}^{\infty} \tfJ_{\ell}^{(n)}(z) \cos(\ell+n)\theta,
\label{inv.3}
\eeq
where $\tfJ_{\ell}^{(n)}(z)=\rmi^\ell\,J_{n+\ell}(z)$ for $\ell\geqslant 0$.
If we put $n=0$, formula \eqref{inv.3} yields:
\beq
e^{\rmi z\cos\theta}=J_0(z) + 2\sum_{\ell=1}^\infty\rmi^\ell \, J_\ell(z)\,\cos\ell\theta,
\label{inv.4}
\eeq
which represents the well-known Jacobi-Anger expansion of a plane wave into a series of 
cylindrical waves \cite[Eq. 10.12.2]{DLMF}.

\skt

\noindent\textit{Case $\nu$ half-integer.}
If $\nu$ is a positive half-integer, $\nu\equiv n+\ud$, $n=0,1,2,\ldots$,
the symmetry formula \eqref{inv.2} reads: 
\beq
\tfJ_{\ell}^{(n+\ud)}(z) = -\tfJ_{-\ell-2n-1}^{(n+\ud)}(z) \qquad (\ell\in\Z; n=0,1,2,\ldots).
\label{d.J.14.bis}
\eeq
For $n=0$ the Fourier coefficients with $\ell<0$ can be obtained from those with $\ell\geqslant 0$
from the symmetry relation: $\tfJ_{\ell}^{(\ud)}(z) = -\tfJ_{-\ell-1}^{(\ud)}(z)$, $(\ell\in\Z)$.
Instead, as in the case of integral $\nu$, if $n\neq 0$ the Fourier coefficients $\tfJ_{\ell}^{(n+\ud)}(z)$
with $-2n\leqslant\ell\leqslant -1$ are expressible as linear functionals of the coefficients 
with $\ell\geqslant 0$.

Now, we can plug formula \eqref{d.J.2} for $\mathfrak{J}^{(n+\ud)}_z(\theta)$ into \eqref{inv.1} and,
in view of  \eqref{d.J.14.bis}, obtain the following expansion for the incomplete gamma function of order $\ud$:
\beq
\gamma(\uds,-\rmi z(1-\cos\theta))
=2\sqrt{\pi}\,\rmi^n\sgn(\theta)e^{-\rmi\frac{\pi}{4}}e^{-\rmi z\cos\theta} \!
\sds\sum_{\ell=-n}^\infty \tfJ_{\ell}^{(n+\ud)}(z)\,\sin(\ell+n+\uds)\theta.
\label{inv.5}
\eeq
Now, putting in \eqref{inv.5} $n=0$, $\varphi=\theta/2$ and $w^2=-2\rmi z$ we obtain:
\beq
\gamma(\uds,w^2\sin^2\varphi)=2\sqrt{\pi}\,\sgn(\varphi)e^{\ud w^2\cos 2\varphi}
\sds\sum_{\ell=0}^\infty (-1)^\ell\,I_{\ell+\ud}(w^2/2)\,\sin(2\ell+1)\varphi,
\label{inv.8}
\eeq
where we used the modified Bessel functions of the first kind $I_\nu(z)=(-\rmi)^\nu J_\nu(\rmi z)$.
Finally, recalling that $\mathrm{erf}(z)=\pi^{-\ud}\gamma(\ud,z^2)$, where $\mathrm{erf}(z)$
denotes the error function \cite[Eq. 7.2.1]{DLMF}, we obtain the following expansion:
\beq
\mathrm{erf}(w\sin\varphi)=2 e^{\ud w^2\cos 2\varphi}
\sds\sum_{\ell=0}^\infty (-1)^\ell I_{\ell+\ud}(w^2/2) \sin(2\ell+1)\varphi
\qquad (-{\textstyle\frac{\pi}{2}}<\mathrm{ph} w \leqslant {\textstyle\frac{\pi}{2}}).
\label{error_function}
\eeq
A similar representation of the error function in terms of modified Bessel
functions of the first kind of integer order is given by Luke in
\cite[Eq. (2.11)]{Luke} through the expansion of the confluent hypergeometric
function in series of Bessel functions.
As a particular case of \eqref{error_function}, we first put $\varphi=\pi/2$ and obtain 
(see also \cite[Eq. (20)]{Luke2}):
\beq
\mathrm{erf}(w)=2 \, e^{-w^2/2}\sum_{\ell=0}^\infty I_{\ell+\ud}(w^2/2).
\label{error_function_2}
\eeq
Similarly, if in \eqref{error_function} we put $\varphi=\pi/4$ and $w\to\sqrt{2}w$, we get the expansion:
\beq
\mathrm{erf}(w)=\sqrt{2} \sum_{\ell=0}^\infty (-1)^\ell\left[I_{2\ell+\ud}(w^2)-I_{2\ell+\frac{3}{2}}(w^2)\right],
\label{error_function_3}
\eeq
which is equivalent to \cite[Eq. 7.6.8]{DLMF} (see also \cite[Eq. (1), p. 122]{Luke2} and 
\cite[Eqs. (33) and (34)]{Veling}).

Finally, as a last example of application of expansion \eqref{inv.1}, we can use the form \eqref{d.J.11}
of $\mathfrak{J}^{(n+\ud)}_z(\theta)$ given in terms of Fresnel integrals. 
We can set $\nu=\ud$ in \eqref{d.J.11} and \eqref{inv.1} and
put $w=2(z/\pi)^\ud$, $a=\cos(\theta-\pi)/2$, then we obtain the following representation of
Fresnel's integral $\mathcal{F}_2(w)\doteq C(w)+\rmi S(w)$ \cite[Eqs. 7.2.7 and 7.2.8]{DLMF} 
(see also \eqref{d.J.11.2}):
\beq
\mathcal{F}_2(aw) = w \, e^{\rmi\frac{\pi}{4}w^2(2a^2-1)}
\sum_{\ell=0}^\infty (-\rmi)^\ell \, T_{2\ell+1}(a) \, j_{\ell}(\pi w^2/4),
\label{error_function_4}
\eeq
where $T_k(\cdot)$ denotes the Chebyshev polynomials of the first kind and $j_\ell(w)$ are
the spherical Bessel functions of the first kind: $j_\ell(w)=\sqrt{\ud\pi/w}\,J_{\ell+\ud}(w)$.

\section{Neumann series of Bessel functions}
\label{se:neumann}

Neumann series of Bessel functions are defined by \cite[Chapter XVI]{Watson}
\beq
\mathfrak{N}_\nu(z) \sim \sum_{n=0}^\infty a_n \, J_{\nu+n}(z),
\label{ne.1}
\eeq
where $z$ is in general a complex variable, and $\nu$ and $\{a_n\}_{n=0}^\infty$ are constants.
This kind of series have been investigated extensively in view of their relevance in a number of 
physical problems (see the Introduction of Ref. \cite{Pogany} for a brief summary of these problems).
In addition, Neumann series have been shown to be a useful mathematical tool to the solution of 
classes of differential and mixed differences equations (see, e.g., \cite{Kravchenko} and \cite[p. 530]{Watson}).
The domain of convergence of series \eqref{ne.1} is a disk whose radius evidently depends on the asymptotic
behaviour of the sequence of coefficients $a_n$ and can be determined by the condition
\beq
\varlimsup_{n\to\infty}\left|\frac{a_n\,(z/2)^{\nu+n}}{\Gamma(\nu+n+1)}\right|^{1/n} < 1.
\label{ne.1.1}
\eeq
Integral representations are powerful tools to deduce properties of series.
In the case of expansions of type \eqref{ne.1}, examples of integral representations have been given 
by Wilkins \cite{Wilkins1}, Rice \cite{Rice} and, more recently, 
by Pog\'any, S\"uli and coworkers \cite{Baricz,Jankov,Pogany}.
In this section we give an integral representation of the function $\mathfrak{N}_\nu(z)$, 
which the Neumann series converges to. 
This is achieved by exploiting the representation \eqref{d.J.1} of $J_{\nu+n}(z)$ given is Section 
\ref{se:besselfunctions}, and assuming the condition on the asymptotic behavior 
of the sequence $\{a_n\}$ which basically guarantees the boundedness of trigonometric series.
The Fourier-type integral representation of $J_{\nu+n}(z)$ allows us to write
$\mathfrak{N}_\nu(z)$ in integral form with a kernel (the kernel associated with Neumann expansions) 
which is proportional to the regularized incomplete gamma function $P(\{\nu\},\cdot)$ and is
independent of the coefficients $a_n$. The density function of the integral representation, instead,
depends only upon the given \emph{input} sequence of coefficients $a_n$,
and is simply the analytic function, on the unit circle, associated with the sequence of 
coefficients $\{a_n\}_{n=0}^\infty$.
This representation easily leads to closed-form integral representations of Neumann expansions of
Bessel functions whenever the closed-form of the analytic function associated with the sequence of coefficients
is known.
As examples of application of this result, new integral representations for
bivariate Lommel's functions and Kelvin's functions will then be given.
More examples of this type of integral representations of special functions of the mathematical physics
will be presented and analyzed in a forthcoming paper \cite{DeMicheli2}.
Now, we can prove the following theorem.

\begin{theorem}
\label{pro:ne.1}
Let $\{a_n\}_{n=0}^\infty \subset\C$ be a null sequence such that
\beq
\sum_{n=0}^\infty |a_n| < \infty.
\label{ne.0.1}
\eeq
Then, for $\Real\nu>-\ud$ the series \eqref{ne.1} has the following integral representation in the slit domain 
$z\in\C\setminus(-\infty,0]$ if $\nu\not\in\N_0$ or in $z\in\C$ if $\nu\in\N_0$:
\beq
\mathfrak{N}_\nu(z) = \frac{1}{2\pi}\int_{-\pi}^\pi K^{(\nu)}_z(\theta)
\,A(e^{\rmi(\theta-\pi/2)})\,\rmd\theta
\qquad \left(\Real\nu>-\uds\right),
\label{ne.2}
\eeq
where 
\beq
A(e^{\rmi\theta}) = \sum_{n=0}^\infty a_n \, e^{\rmi n\theta},
\label{ne.4}
\eeq
and the kernel $K^{(\nu)}_z(\theta)$ is given by
\beq
K^{(\nu)}_z(\theta) = \rmi^\nu\,e^{\rmi\nu[\theta-\pi\sgn(\theta)]}e^{\rmi z\cos\theta} 
\, P(\{\nu\},-\rmi z\,(1-\cos\theta)),
\label{ne.3}
\eeq
$P(\cdot)$ denoting the regularized gamma function.
\end{theorem}

\begin{proof}
We first insert the integral representation \eqref{d.J.1} of $J_{\nu+n}(z)$ into sum \eqref{ne.1}:
\beq
\sum_{n=0}^\infty a_n \, J_{\nu+n}(z)
=\sum_{n=0}^\infty (-\rmi)^n a_n \int_{-\pi}^\pi \mathfrak{J}^{(\nu)}_z(\theta) \, e^{\rmi n\theta}\,\rmd\theta
\qquad (\Real\nu>-\uds),
\label{ne.4.1}
\eeq
and then swap sum and integral to give:
\beq
\int_{-\pi}^\pi\rmd\theta \ \mathfrak{J}^{(\nu)}_z(\theta) \, \sum_{n=0}^\infty (-\rmi)^n a_n\,e^{\rmi n\theta}
= \frac{1}{2\pi}\int_{-\pi}^\pi K^{(\nu)}_z(\theta) \, A(e^{\rmi(\theta-\pi/2)}) \,\rmd\theta,
\label{ne.4.2}
\eeq
where $K^{(\nu)}_z(\theta)$ is given by \eqref{ne.3} and $A(e^{\rmi\theta})$ by \eqref{ne.4}. 
By the Fubini-Tonelli theorem, exchanging sum and integral is legitimate if
\beq
\int_{-\pi}^\pi \left|K^{(\nu)}_z(\theta)\,A(e^{\rmi(\theta-\pi/2)})\right| \,\rmd\theta <\infty.
\label{ne.4.3}
\eeq
For $\theta\in(-\pi,\pi]$ the kernel $K^{(\nu)}_z(\theta)$ in \eqref{ne.3} can be rewritten as follows:
\beq
K^{(\nu)}_z(\theta)
=\rmi^{[\Real\nu]}\,z^{\{\nu\}} \, e^{\rmi\nu[\theta-\pi\sgn(\theta)]}e^{\rmi z\cos\theta}
(1-\cos\theta)^{\{\nu\}}\,\gamma^*(\{\nu\},-\rmi z\,(1-\cos\theta)),
\label{ne.4.4}
\eeq
where $\gamma^*(a,w)$ denotes the Tricomi form of the incomplete gamma function.
Recall that $\gamma^*(a,w)$ is an entire function either in $a$ and in $w$.
Consequently, on compact domains in $\Omega\doteq(z\in\C\setminus(-\infty,0])\times (-\ud<\{\nu\}<1)$
the lower incomplete gamma function $|\gamma^*(\{\nu\},-\rmi z(1-\cos\theta))|$ can be bounded 
uniformly in $\theta$, i.e.: $|\gamma^*(\{\nu\},-\rmi z(1-\cos\theta))|\leqslant c_{z,\nu}$, 
the constant $c_{z,\nu}$ depending only on $z$ and $\nu$. \cite{Gautschi,Tricomi}.
From \eqref{ne.4.4} it is clear that, for fixed $\theta$, the multivaluedness 
of the kernel $K^{(\nu)}_z(\theta)$ is brought by the fractional power $z^{\{\nu\}}$ (when $\{\nu\}\neq 0$).
From \eqref{ne.4.4} we have:
\beq
\left|K^{(\nu)}_z(\theta)\right|\leqslant
c_{z,\nu}\,e^{3\pi|\Imag\nu|}\,e^{|z|}\,|z|^{\Real\{\nu\}} |1-\cos\theta|^{\{\nu\}} 
= C_{\nu,z} \,|1-\cos\theta|^{\{\nu\}},
\label{ne.4.6}
\eeq
$C_{\nu,z}$ being a finite constant on compacta in $\Omega$.
From \eqref{ne.4.3} and \eqref{ne.4.6} it then follows:
\beq
\begin{split}
& \left|\mathfrak{N}_\nu(z)\right|=
\left|\int_{-\pi}^\pi K^{(\nu)}_z(\theta)\,A(e^{\rmi(\theta-\pi/2)})\,\rmd\theta\right|
\leqslant\int_{-\pi}^\pi 
\left|K^{(\nu)}_z(\theta)\right|\,\left|\sum_{n=0}^\infty a_n\,e^{\rmi(\theta-\pi/2)}\right| \,\rmd\theta \\ 
&\quad\leqslant C_{\nu,z} \sum_{n=0}^\infty |a_n| \, \int_{-\pi}^\pi |1-\cos\theta|^{\{\nu\}} \,\rmd\theta < \infty,
\end{split}
\label{ne.4.7}
\eeq
if $\Real\{\nu\}>-\ud$ and in view of \eqref{ne.0.1}. Statement \eqref{ne.2} then follows.
\end{proof}

\begin{remark}
From \eqref{ne.4.6} we see that for $\Real\nu\geqslant 0$ the kernel $K_z^{(\nu)}(\theta)$ is bounded uniformly
in $\theta$ on compacta in $\Omega$. Hence, assumption \eqref{ne.0.1} of boundedness 
for $A(e^{\rmi\theta})$ can be substituted by a less restrictive assumption of integrability (see \eqref{ne.4.7}). 
Explicitly, this is guaranteed for instance (see \cite[Theorem 1, p. 54]{Bray}) by replacing 
\eqref{ne.0.1} with the weaker asymptotic condition  
\beq
\sum_{n=1}^\infty\frac{|a_n|}{n} < \infty,
\label{ne.0.1.bis}
\eeq
along with the following additional regularity assumption on the sequence $\{a_n\}$:
\beq
\sum_{n=1}^\infty\left(\frac{\sum_{k=n}^\infty\left|\Delta a_k\right|^p}{n}\right)^{1/p}<\infty
\qquad \mathrm{for~some~integer~}p>1,
\label{ne.0}
\eeq
where $\Delta a_k=(a_k-a_{k+1})$.
\end{remark}

\begin{remark}
The integral representation \eqref{ne.2} can also be written in the following form:
\beq
\mathfrak{N}_\nu(z) 
=\frac{e^{\rmi z}}{\pi} \, \int_{0}^\pi \mathcal{K}^{(\{\nu\})}(-\rmi z(1+\cos\theta)) 
\ \mathcal{A}^{(\nu)}(\theta)\,\rmd\theta,
\label{ne.2a}
\eeq
where the kernel is
\beq
\mathcal{K}^{(\{\nu\})}(w) = e^w \, P(\{\nu\},w), 
\label{ne.2b}
\eeq
and the function associated with the sequence of coefficients $\{a_n\}$ is:
\beq
\mathcal{A}^{(\nu)}(\theta) = \sum_{n=0}^\infty \rmi^{n+\nu} a_n \, \cos(n+\nu)\theta.
\label{ne.2c}
\eeq
\end{remark}

\subsection{Lommel's functions of two variables}
\label{subse:lommel}

As first example of application of the previous analysis, we now consider the bivariate Lommel 
functions $U_\nu(w,z)$ and $V_\nu(w,z)$ of order $\nu$. These functions can be regarded as 
particular solutions of the inhomogeneous Bessel equation (see \cite[Section 27]{Korenev} 
for a summary of their properties) and were first introduced by Lommel
in the analysis of diffraction problems \cite[p. 438]{Born}.
They have found applications also, e.g., in soliton theory \cite{Flesch}
and in the theory of incomplete cylindrical functions \cite[Chapter 4]{Agrest}.
Lommel's functions of two variables can be defined for unrestricted values of $\nu$ by the Neumann 
series \cite[Section 16.5, Eqs. (5) and (6); p. 537, 538]{Watson}:
\begin{align}
U_\nu(w,z) & \doteq 
\sum_{n=0}^\infty (-1)^n \left(\frac{w}{z}\right)^{\nu+2n} J_{\nu+2n}(z), 
\label{lomm.1} \\
V_\nu(w,z) & \doteq \cos\left(\frac{w}{2}+\frac{z^2}{2w}+\frac{\nu\pi}{2}\right) + U_{-\nu+2}(w,z).
\label{lomm.2}
\end{align}
The function $\left(w/z\right)^{-\nu} U_\nu(w,z)$ has therefore the structure \eqref{ne.1} with coefficients:
\beq
a_n(w,z) = 
\begin{cases}
(-1)^{n/2}(w/z)^n & \textrm{if} \ n \ \textrm{is even}, \\[+2pt]
0 & \textrm{if} \ n \ \textrm{is odd}.
\end{cases}
\label{lomm.3}
\eeq
Condition \eqref{ne.0.1} then reads
\beq
\sum_{n=0}^\infty |a_n(w,z)| = \sum_{n=0}^\infty |w/z|^{2n}<\infty,
\label{lomm.4}
\eeq
if $|w/z|<1$.
Therefore, assuming the condition $|w/z|<1$ and using \eqref{lomm.3}, 
the analytic function $A_{w,z}^{(U)}(e^{\rmi\theta})$ associated with the bivariate Lommel function $U_\nu(w,z)$ 
then reads (see \eqref{ne.4}):
\beq
A_{w,z}^{(U)}(e^{\rmi\theta}) = \sum_{n=0}^\infty (-1)^n \left(\frac{w}{z}\right)^{2n}e^{\rmi 2n\theta}, 
\label{lomm.6}
\eeq
which is a power series converging for $|w/z|<1$ and any $\theta\in(-\pi,\pi]$ to
\beq
A_{w,z}^{(U)}(e^{\rmi\theta})= \frac{1}{1+(w/z)^2e^{\rmi 2\theta}}.
\label{lomm.7}
\eeq
Finally, from \eqref{ne.2} the integral representation of the Lommel function $U_\nu(w,z)$ reads explicitly
for $\Real\nu>-\ud$ and $|w/z|<1$:
\beq
\begin{split}
& U_\nu(w,z) \\
&\quad = \frac{\rmi^\nu}{\pi}\!\left(\frac{w}{z}\right)^{\!\nu} \!\!
\int_0^\pi \frac{P(\{\nu\},-\rmi z(1+\cos\theta))\,[\cos\nu\theta-(w/z)^2\cos(\nu-2)\theta]}
{1-2(w/z)^2\cos2\theta+(w/z)^4}e^{-\rmi x\cos\theta}\,\rmd\theta.
\end{split}
\label{lomm.8}
\eeq
In view of the restriction $\Real\nu>-\ud$, a similar integral representation can be
obtained from \eqref{lomm.2} for the Lommel function $V_\nu(w,z)$ for $\Real\nu<\frac{5}{2}$.

\skt

\subsection{Kelvin's functions}
\label{subse:kelvin}

Kelvin's functions of the first kind $\mathrm{ber}_\nu(x)$ and $\mathrm{bei}_\nu(x)$ are, respectively,
real and imaginary part of the solution (the symbols stand indeed for \emph{Bessel-real} and 
\emph{Bessel-imaginary}), regular in the origin, of the differential equation \cite{Apelblat}
\beq
x^2 y'' + xy' -(\nu^2+\rmi x^2)y = 0,
\label{kel.1}
\eeq
and emerge as solutions of various physics and engineering problems occurring in the theory of
electrical currents \cite{Brualla}, magnetism \cite{Jardim}, fluid mechanics and elasticity \cite{Apelblat}. 
Kelvin's functions of order $\nu$ are related to the Bessel functions as follows \cite{DLMF}:
\beq
\mathrm{ber}_\nu(x) + \rmi\,\mathrm{bei}_\nu(x) = e^{\rmi\pi\nu}J_\nu(xe^{-\rmi\pi/4})
\qquad (x\geqslant 0,\nu\in\R).
\label{kel.2}
\eeq
Integral representations are known: for instance, when $\nu\in\R$, Schl\"{a}fli's representation of
Bessel's function $J_\nu(x)$ leads to the following representation of Kelvin's \cite{Apelblat}:
\begin{align}
\begin{split}
\mathrm{ber}_\nu(\sqrt{2}\,x) &= \frac{1}{\pi}\int_0^\pi [\cos\pi\nu\cos(x\sin t-\nu t)\cosh(x\sin t) \\
&\qquad\qquad-\sin\pi\nu\sin(x\sin t-\nu t)\sinh(x\sin t)]\,\rmd t \\
&\quad-\frac{\sin\pi\nu}{\pi}\int_0^\infty e^{-\nu t-x\sinh t} \cos(x\sinh t+\pi\nu)\,\rmd t,
\end{split}
\label{kel.3a} \\
\begin{split}
\mathrm{bei}_\nu(\sqrt{2}\,x) &= \frac{1}{\pi}\int_0^\pi[\cos\pi\nu\sin(x\sin t-\nu t)\sinh(x\sin t) \\
&\qquad\qquad+\sin\pi\nu\cos(x\sin t-\nu t)\cosh(x\sin t)]\,\rmd t \\
&\quad-\frac{\sin\pi\nu}{\pi}\int_0^\infty e^{-\nu t-x\sinh t}\sin(x\sinh t+\pi\nu)\,\rmd t, 
\end{split}
\label{kel.3b}
\end{align}
which, when $\nu\equiv n$ is a nonnegative integer, reduce to:
\begin{subequations}
\begin{align}
\mathrm{ber}_n(\sqrt{2}\,x) &= \frac{(-1)^n}{\pi}\int_0^\pi\cos(x\sin t-nt)\cosh(x\sin t)\,\rmd t,
\label{kel.4a} \\
\mathrm{bei}_n(\sqrt{2}\,x) &= \frac{(-1)^n}{\pi}\int_0^\pi\sin(x\sin t-nt)\sinh(x\sin t)\,\rmd t.
\label{kel.4b}
\end{align}
\label{kel.4}
\end{subequations}
A novel integral representation for $\mathrm{ber}_\nu(x)+\rmi\,\mathrm{bei}_\nu(x)$ can be readily 
obtained from the results of Theorem \ref{pro:ne.1} (see \eqref{ne.2}) observing that, for $x\in\R$ and
$\nu\in\C$, the following expansion in series of Bessel functions holds \cite[Eq. 10.66.1]{DLMF}:
\beq
\mathrm{ber}_\nu(x) + \rmi\,\mathrm{bei}_\nu(x)
= e^{\rmi 3\pi\nu/4}\sum_{k=0}^\infty
\frac{x^k}{2^{k/2}\,k!} \,e^{\rmi k\pi/4} \, J_{\nu+k}(x).
\label{kel.5}
\eeq
Referring to \eqref{ne.1}, we set
\beq
a_k (x) = \frac{x^k}{2^{k/2}\,k!}\,e^{\rmi k\pi/4} \qquad (k\geqslant 0).
\label{kel.6}
\eeq
Then, from \eqref{ne.4} the analytic function associated with Kelvin's functions reads:
\beq
A_x^{(K)}(\theta)
= \sum_{k=0}^\infty \frac{x^k}{2^{k/2}\,k!}\, e^{\rmi k(\theta+\pi/4)}
=\exp(x\,e^{\rmi(\theta+\pi/4)}/\sqrt{2}),
\label{kel.7}
\eeq
the series converging for any $x\in\R$ and $\theta\in(-\pi,\pi]$.
Finally, using \eqref{ne.2}, \eqref{ne.3} and \eqref{kel.7}
we have for $\Real\nu>-\ud$ the following integral representation for Kelvin's functions:
\beq
\begin{split}
& \mathrm{ber}_\nu(\sqrt{2}x) + \rmi\,\mathrm{bei}_\nu(\sqrt{2}x) \\
& \qquad = \frac{e^{\rmi 5\pi\nu/4}}{\pi}\int_{0}^\pi 
e^{-cx\cos\theta}\,P(\{\nu\},-\rmi\sqrt{2}x(1+\cos\theta))\,
\cos\left[\nu\theta-\overline{c}x\sin\theta\right]\,\rmd\theta,
\end{split}
\label{kel.10}
\eeq
with $c \doteq \exp(\rmi\pi/4)$ and the bar standing for complex conjugate. 
Formula \eqref{kel.10} represents the generalization to $\Real\nu>-\ud$ of formulae \eqref{kel.4}.

\skt\skt

\bibliographystyle{amsplain}

\begin{thebibliography}{99}

\bibitem{Agrest}
M.M Agrest, M.S. Maksimov,
Theory of Incomplete Cylindrical Functions and their Applications,
Springer-Verlag, Berlin (1971).

\bibitem{Apelblat}
A. Apelblat,
Integral representation of Kelvin functions and their derivatives with respect to the order,
J. Appl. Math. Phys. (ZAMP) \textbf{42}(5) (1991) 708-714.

\bibitem{Baricz}
\'A. Baricz, D. Jankov, T.K. Pog\'any,
Neumann series of Bessel functions,
Integral Transf. Spec. Funct. \textbf{23}(7) (2012), 529-538.

\bibitem{Born}
M. Born, E. Wolf,
Principles of Optics,
Pergamon Press, Oxford, 1965.

\bibitem{Bray}
W.O. Bray, V.B. Stanojevi\'c,
On the integrability of complex trigonometric series,
Proc. Amer. Math. Soc. \textbf{93}(1) (1985), 51-58.

\bibitem{Bros}
J. Bros, G.A. Viano,
Connection between the harmonic analysis on the sphere and the harmonic
analysis on the one--sheeted hyperboloid: an analytic continuation viewpoint III.
Forum Math. \textbf{9} (1997), 165-191.

\bibitem{Brualla}
L. Brualla, P. Martin,
Analytic approximations to Kelvin functions with applications to electromagnetics,
J. Phys. A: Math. Gen. \textbf{34}(43) (2001), 9153.

\bibitem{DeMicheli3}
E. De Micheli, G.A. Viano,
Holomorphic extension associated with Fourier-Legendre expansions,
J. Geom. Anal. \textbf{12}(3) (2002), 355-374.

\bibitem{DeMicheli1}
E. De Micheli, G.A. Viano,
The expansion in Gegenbauer polynomials: A simple method for the fast computation of the Gegenbauer coefficients,
J. Comput. Phys. \textbf{239} (2013), 112-122.

\bibitem{DeMicheli2}
E. De Micheli,
On the integral representation of special functions related to Bessel's functions of the first kind,
in preparation, 2017.

\bibitem{DLMF}
NIST Digital Library of Mathematical Functions. http://dlmf.nist.gov/, Release 1.0.15. 
F.W.J. Olver, A.B. Olde Daalhuis, D.W. Lozier, B.I. Schneider, R.F. Boisvert, C.W. Clark, 
B.R. Miller, and B.V. Saunders, eds.

\bibitem{Drascic}
B. Dra\u{s}\u{c}i\'c, T.K. Pog\'any,
On integral representation of Bessel function of the first kind,
J. Math. Anal. Appl. \textbf{308}(2) (2005), 775-780.

\bibitem{Flesch}
R.J. Flesch, S.E. Trullinger,
Green's functions for nonlinear Klein-Gordon kink perturbation theory,
J. Math. Phys. \textbf{28}(7) (1987), 1619-1631

\bibitem{Gautschi}
Gautschi, W.:
The incomplete gamma functions since Tricomi,
in \emph{Tricomi's Ideas and Contemporary Applied Mathematics}, 
Atti dei Convegni Lincei \textbf{147}, Accademia Nazionale dei Lincei (1998), 203-237.

\bibitem{Jankov}
D. Jankov, T.K. Pog\'any, E. S\"uli,
On the coefficients of Neumann series of Bessel functions,
J. Math. Anal. Appl. \textbf{380}(2) (2011), 628-631.

\bibitem{Jardim}
R.F. Jardim, B. Laks,
Kelvin functions for determination of magnetic susceptibility in nonmagnetic metals,
J. Appl. Phys. \textbf{65}(12) (1989), 4505.

\bibitem{Korenev}
B.G. Korenev,
Bessel Functions and their Applications,
Chapman \& Hall/CRC, Boca Raton (FL), 2002.

\bibitem{Kravchenko}
V.V. Kravchenko, S.M. Torba, R. Castillo-P\`erez,
A Neumann series of Bessel functions representation for solutions of perturbed Bessel equations,
Applicable Analysis, DOI:10.1080/00036811.2017.1284313, (2017).

\bibitem{Luke}
Y.L. Luke,
Expansion of the confluent hypergeometric function in series of Bessel functions,
Mathematical tables and other aids to computation \textbf{13} (1959), 261-271.

\bibitem{Luke2}
Y.L. Luke,
The Special Functions and their Approximations, Vol. 2,
Academic Press, New York (NY), 1969.

\bibitem{Pogany}
T.K. Pog\'any, E. S\"uli,
Integral representation of Neumann series of Bessel functions,
Proc. Amer. Math. Soc. \textbf{137}(7) (2009), 2363-2368.

\bibitem{Rice}
S.O. Rice, 
Mathematical analysis of random noise, III, 
Bell Syst. Tech. J. \textbf{24}(1) (1945), 46–156.

\bibitem{Szego}
G. Szeg\"o,
Orthogonal Polynomials,
Amer. Math. Soc., Providence (RI), 1975.

\bibitem{Tricomi}
F.G. Tricomi,
Asymptotische Eigenschaften der unvollst\"{a}ndigen Gammafunktion,
Math. Z. \textbf{53}(2) (1950), 136–-148.

\bibitem{Veling}
E.J.M. Veling,
The generalized incomplete gamma function as sum over modified Bessel functions of the first kind,
J. Comput. Appl. Math. \textbf{235}(14) (2011), 4107-4116.

\bibitem{Vilenkin}
N.I. Vilenkin,
Special Functions and the Theory of Group Representations,
Transl. Math. Monogr. \textbf{22}, Amer. Math. Soc., Providence (RI), 1968.

\bibitem{Watson}
G.N. Watson,
A Treatise on the Theory of Bessel Functions,
Cambridge University Press, Cambridge, 1922.

\bibitem{Wilkins1}
J.E. Wilkins, Jr.,
Neumann series of Bessel functions,
Trans. Amer. Math. Soc. \textbf{64}(2) (1948), 359-385.

\end{thebibliography}

\end{document}